\documentclass{article}
\usepackage{amsthm}
\usepackage{amssymb}
\usepackage{amsmath}
\usepackage{epsfig}

\usepackage{graphicx}
\usepackage{epsfig}
\usepackage{pstricks}

\newtheorem{thm}{Theorem}[section]

\newtheorem{lemma}[thm]{Lemma}

\usepackage{amsmath}
\usepackage{amsxtra}
\usepackage{amscd}
\usepackage{amsthm}
\usepackage{amsfonts}
\usepackage{amssymb}
\usepackage{eucal}
\newcommand{\bmb}{\left( \begin{array}{rr}}
\newcommand{\enm}{\end{array}\right)}

\newcommand{\C}{{\mathbb C}}
\newcommand{\Z}{{\mathbb Z}}

\newcommand{\N}{{\mathbb N}}

\newcommand{\al}{{\alpha}}

\numberwithin{equation}{section}

\begin{document}

\vspace*{1.5pc}

\large

{\bf\Large \begin{center}
Truncated determinants and the refined enumeration of Alternating Sign Matrices and Descending Plane Partitions
\end{center}}

\medskip

\centerline{\Large 
Philippe Di Francesco}

\medskip

\centerline{\Large 
Institut de Physique Th\'eorique,}
\centerline{\Large Commissariat \`a l'Energie Atomique, Saclay, France}

\bigskip
\medskip

\section{Introduction}

In these notes, we will be mainly focussing on the proof of the so-called ASM-DPP conjecture of
Mills, Robbins and Rumsey \cite{MRR} which relates refined enumerations of Alternating Sign Matrices
(ASM) and Descending Plane Partitions (DPP).

ASMs were introduced by Mills, Robbins and Rumsey \cite{RR} in their study of Dodgsonﾕs condensation algorithm 
for the evaluation of determinants. DPPs were introduced by  Andrews \cite{AN} while attempting to prove a conjectured 
formula for the generating function of cyclically symmetric plane partitions. 

\subsection{ASMs from lambda-determinant}

The definition of the so-called lambda-determinant of Mills, Robbins and Rumsey \cite{RR} is based on the famous
Dodgson condensation algorithm \cite{DOD} for computing determinants, itself based on the Desnanot-Jacobi
equation, a particular Pl\"ucker relation, relating minors of any square $k+1\times k+1$ matrix $M$:
\begin{equation}\label{desnajac}
\vert M\vert \times \vert M_{1,k+1}^{1,k+1}\vert=\vert M_{k+1}^{k+1}\vert \times \vert M_1^1\vert -
\vert M_1^{k+1}\vert \times \vert M_{k+1}^1\vert
\end{equation}
where $\vert M_{i_1,i_2,...,i_r}^{j_1,j_2,...,j_r}\vert$ stands for the determinant of the matrix obtained from $M$
by deleting rows $i_1,...,i_r$ and columns $j_1,...,j_r$. The relation \eqref{desnajac} may be used as a
recursion relation on the size of the matrix, allowing for efficiently compute its determinant.

More formally, we may recast the algorithm using the so-called $A_\infty$ $T$-system 
(also known as discrete Hirota) relation:
\begin{equation}\label{tsys}
T_{i,j,k+1}T_{i,j,k-1}=T_{i,j+1,k}T_{i,j-1,k}-T_{i+1,j,k}T_{i-1,j,k}
\end{equation}
for any $i,j,k\in \Z$ with fixed parity of $i+j+k$. 
Now let $A=(a_{i,j})_{i,j\in \{1,2,...,n\}}$ be a fixed $n\times n$ matrix. Together with the initial data:
\begin{eqnarray}
T_{\ell,m,0}&=&1\quad \qquad \qquad\qquad (\ell,m\in \Z;\ell+m=n\, {\rm mod}\, 2)\nonumber \\ 
T_{i,j,1}&=&a_{{j-i+n+1\over 2},{i+j+n+1\over 2}} \quad
(i,j\in\Z;i+j=n+1\, {\rm mod}\, 2; |i|+|j|\leq n-1)\, ,\label{initdat}
\end{eqnarray}
the solution of the $T$-system \eqref{tsys} satisfies:
\begin{equation}
T_{0,0,n}=\det(A)
\end{equation}

Given a fixed formal parameter $\lambda$,
the lambda-determinant of the matrix $A$, denoted by $\vert A\vert_\lambda$ is simply defined 
as the solution $T_{0,0,n}=\vert A\vert_\lambda$ of the deformed $T$-system
\begin{equation}\label{defoT} 
T_{i,j,k+1}T_{i,j,k-1}=T_{i,j+1,k}T_{i,j-1,k}+\lambda \, T_{i+1,j,k}T_{i-1,j,k}
\end{equation}
subject to the initial condition \eqref{initdat}.

The discovery of Mills, Robbins and Rumsey is that the lambda-determinant is a homogeneous
Laurent polynomial of the matrix entries of degree $n$, and that moreover the monomials
in the expression are coded by $n\times n$ matrices $B$ with entries $b_{i,j}\in \{0,1,-1\}$, characterized by the fact that
their row and column sums are $1$ and that the partial row and column sums are non-negative, namely
\begin{eqnarray*} \sum_{i=1}^k b_{i,j} \geq 0 & & \quad \sum_{i=1}^k b_{j,i}\geq 0\quad (k=1,2,...,n-1;j=1,2,...,n)\\
 \sum_{i=1}^n b_{i,j} = 1 & & \quad \sum_{i=1}^n b_{j,i}=1\quad (j=1,2,...,n)
\end{eqnarray*}
Such matrices $B$ are called alternating sign matrices (ASMs). These include the permutation matrices (the ASMs
with no $-1$ entry). Here are the 7 ASMs of size 3:
$$\small  \begin{pmatrix}1 &0&0\\0&1&0\\0&0&1\end{pmatrix}\quad
\begin{pmatrix}1 &0&0\\0&0&1\\0&1&0\end{pmatrix}\quad
\begin{pmatrix}0 &1&0\\1&0&0\\0&0&1\end{pmatrix}\quad
\begin{pmatrix}0&1&0\\1&-1&1\\0&1&0\end{pmatrix}\quad
\begin{pmatrix}0 &1&0\\0&0&1\\1&0&0\end{pmatrix}\quad
\begin{pmatrix}0 &0&1\\1&0&0\\0&1&0\end{pmatrix}\quad
\begin{pmatrix}0 &0&1\\0&1&0\\1&0&0\end{pmatrix}$$
There is an explicit formula for the lambda-determinant\cite{RR}:
\begin{equation}\label{lambdadet}
\vert A\vert_\lambda=\sum_{n\times n\, ASM\, B} \lambda^{{\rm Inv}(B)-N(B)} (1+\lambda)^{N(B)}\prod_{i,j} a_{i,j}^{b_{i,j}}
\end{equation}
where ${\rm Inv}(B)$ and $N(B)$ denote respectively the inversion number and the number of entries $-1$ in $B$,
with 
\begin{eqnarray*}
{\rm Inv}(B)&=&\sum_{1\leq i<j \leq n\atop 1\leq k<\ell \leq n} b_{i,\ell}b_{j,k} \\
N(B)&=&\frac{1}{2}\left( -n+\sum_{1\leq i,j\leq n} |b_{i,j}| \right) 
\end{eqnarray*}
Note that for $\lambda=-1$, only the ASMs with $N(B)=0$ contribute, i.e. the permutation matrices, for which
${\rm Inv}(B)$ coincides with the usual inversion number of the corresponding permutation, and therefore \eqref{lambdadet}
reduces to the usual formula for the determinant. 

Mills, Robbins and Rumsey \cite{MRR}
noticed that apart from the quantities ${\rm Inv}(B)$ and $N(B)$, another
``observable" of interest is the position of the unique $1$ in the top row of any ASM. We denote
by $t(B)$ the number of $0$ entries to the left of the $1$ in the top row of $B$.

Associating a weight 
\begin{equation}\label{weightasm}
W(B)=z^{t(B)}y^{{\rm Inv}(B)-N(B)} x^{N(B)}\end{equation} 
to each ASM $B$, we may form the partition function
\begin{equation}\label{partASM}
Z_{ASM}^{(n)}(x,y,z)=\sum_{n\times n\, ASM\, B} W(B) \end{equation}

In the case $n=3$ listed above, the ASMs receive respective weights: 
$1,zy,zxy,y,zy^2,z^2y^2,z^2y^3$, leading to $Z_{ASM}^{(3)}(x,y,z)=1+zy+zxy+y+zy^2+z^2y^2+z^2y^3$.

\subsection{DPPs}

Descending plane partitions are arrays of positive integers of the form:
$$\begin{matrix} a_{1,1} & a_{1,2} & a_{1,3} & \cdots & \cdots & a_{1,\mu_1-2} & a_{1,\mu_1-1} & a_{1,\mu_1} \\
 & a_{2,2} & a_{2,3} & \cdots & \cdots & a_{2,\mu_2} & & \\
 & & a_{3,3} & \cdots & a_{3,\mu_3} & & & \\
 & & \ddots & \cdots & & & & \\
 & & & a_{r,r}\cdots a_{r,\mu_r} & & & &
\end{matrix}$$
such that the sequece $\mu_i$ is strictly decreasing $\mu_{i+1}<\mu_i$, and that for $\lambda_i=\mu_i-i+1$,
$\lambda_0=\infty$:
$$  a_{i,j}\geq a_{i,j+1} \qquad a_{i,j}>a_{i+1,j} \qquad \lambda_i<a_{i,i}\leq \lambda_{i-1} $$
for all $i,j$. By convention, the empty partition is a DPP.
Here are the 7 DPPs of order 3:
$$\emptyset,\qquad \begin{matrix} 2 \end{matrix},\qquad \begin{matrix} 3 \end{matrix}, \qquad 
\begin{matrix} 3 & 1 \end{matrix},\qquad 
\begin{matrix} 3 & 2\end{matrix},\qquad \begin{matrix} 3 & 3\end{matrix},
\qquad \begin{matrix} 3 & 3\\ & 2\end{matrix}\, .$$

The integers $a_{i,j}$ are called parts. A DPP $A$ is said to be of order $n$ if $a_{i,j}\leq n$ for all $i,j$. A part
$a_{i,j}$ is said to be special if $a_{i,j}\leq j-i$.
We denote by $S(A)$ and $NS(A)$ respectively the total number of special parts and the total
number of non-special parts of any DPP $A$. Another observable of interest among the DPPs $A$ of order $n$ 
is the number of parts in $A$ equal to the order, which we denote by $M(A)$. 
To each DPP $A$ of given order, we associate a weight
\begin{equation}\label{weightdpp}
W(A)=x^{S(A)}y^{NS(A)}z^{M(A)}
\end{equation}
and define the partition function for DPPs of order $n$ to be:
\begin{equation}\label{partdpp}
Z_{DPP}^{(n)}(x,y,z)=\sum_{DPP\, A\, {\rm of}\, {\rm order}\, n} W(A) 
\end{equation}

The 7 DPPs of order $3$ listed above have respective weights: 
$1,y,zy,zxy,zy^2,z^2y^2,z^2y^3$, as $M(A)$ is the number of occurrences of the part $3$,
and the only special part is the entry $1$ in the fourth DPP.
This leads to the partition function $Z_{DPP}^{(3)}= 1+y+zy+zxy+zy^2+z^2y^2+z^2y^3$.

\subsection{The ASM-DPP conjecture}

The ASM-DPP conjecture as stated by Mills, Robbins and Rumsey \cite{MRR} amounts to the
identity between the partition functions of ASMs and DPPs as defined in the previous sections.
This is the following:
\begin{thm}\label{mrrthm}
The partition functions for the refined enumeration of ASMs and DPPs coincide, namely
$$Z_{ASM}(x,y,z)=Z_{DPP}(x,y,z)$$
\end{thm}

This was finally proved in all its generality in \cite{BDZJ1}, and then generalized so as to include
yet another observable in \cite{BDZJ2}. In the present note, we explain the rationale behind
these proofs which strongly rely on manipulations of finite truncations of infinite matrices.

For simplicity of exposition we shall start with the identity between the doubly refined partition functions
$Z_{ASM}^{(n)}(x,y,1)$ and $Z_{DPP}^{(n)}(x,y,1)$. Each will be expressed as the determinant of a finite
truncation (of size $n\times n$) of an infinite matrix, and the identity between determinants will be
derived from general principles relating the two ``infinite" matrices. One key ingredient is the use of
the double generating series for the matrix entries (see Appendix \ref{genapp} for definitions and properties).

\subsection{Outline}

The use of infinite matrices is somewhat
non-standard in this context, and we would like to stress the power and beauty of the method.
The infinite matrices occurring here actually involve some fundamental object that came up in the study of so-called
Lorentzian triangulations \cite{LORGRA}, giving rise to one of the simplest examples of quantum
integrable system. More precisely, random configurations of this particular class of triangulations may be generated by
iterated powers of a transfer matrix $T$ of  infinite size. The problem was solved exactly by diagonalization
of $T$ in \cite{LORGRA}. A drastic simplification of the problem comes from the existence
of  an infinite parametric family of such transfer matrices, which all commute with each other. 

The notes are organized as follows.

In Section 2, we recall a number of facts about the transfer matrix of $1+1$-dimensional Lorentzian
triangulations, including other applications to trees and lattice path enumeration. 

In Section 3, we
compute $Z_{ASM}^{(n)}(x,y,1)$ by use of the Izergin-Korepin (IK) \cite{IZ,KO} determinant formulation of the
partition function of the bijectively equivalent configurations of the 6 Vertex (6V) model with
Domain-Wall Boundary Conditions (DWBC). The difficulty here is to extract a homogeneous limit out of
the IK determinant, and to put it in the form of the determinant of a finite truncation to size $n\times n$ 
of an infinite matrix which is independent of $n$. 

In Section 4, we compute $Z_{DPP}^{(n)}(x,y,1)$
by use of the lattice path formulation of the problem \cite{LA}, and by use of the
Lindstr\"om Gessel-Viennot (LGV) determinant formula for the partition function of
non-intersecting families of lattice paths. This expresses  $Z_{DPP}^{(n)}(x,y,1)$ as the determinant
of the finite truncation to size $n\times n$ of an infinite matrix independent of $n$. 

In Section 5, we show how the relation between the double generating functions
of the two infinite matrices above implies the identity between the determinants
of any finite truncation thereof. This is the key to the proof of the ASM-DPP conjecture. 
We then show how this has to be adapted to include more refinements.

In the conclusive Section 6, the ASM-DPP correspondence is placed in the wider context
of the myriad of combinatorial objects and physical systems connected to ASMs.
We also compare the two very different forms of quantum integrability underlying the ASM-DPP correspondence,
one coming from the Lorentzian triangulations, the other from the 6V model. 

We collect the useful formulas and definitions for generating functions and truncated determinants of infinite
matrices in Appendix \ref{genapp}.

\noindent{\bf Acknowledgments.} These notes are largely based on work with E. Guitter, C. Kristjansen,
R. Behrend and P. Zinn-Justin. 
I thank the Mathematical Sciences Research Instiitute, Berkeley, California
for hosting me while these notes  were completed.

\section{The main actor: the transfer matrix of Lorentzian gravity}

\subsection{1+1D Lorentzian gravity}
Discrete models for 1+1D Lorentzian gravity are defined as follows.  They are statistical models
whose configurations are discrete space-times, in the form of random triangulations
with a regular discrete time direction (an integer segment $[t_1,t_2]\subset \N$) 
and a random space direction, modeled
by random triangulations of the unit time strips $[t,t+1]$, $t\in \N$, by arbitrary but finite numbers of triangles with
one edge along the time line $t$ (resp. $t+1$) and the opposite vertex on the time line $t+1$ (resp. $t$).
All other edges are then glued to their neighbors so as to form a triangulation. 
Each ``horizontal edge" along a time line $t$ is shared by two triangles, one in each time slice $[t-1,t]$ and $t,t+1]$.
The boundary conditions along the time lines may be taken
free, periodic or staircase-like \cite{LORGRA}.
A typical free boundary Lorentzian triangulation $\Theta$  in 1+1D reads as follows:
$$ \raisebox{-1.cm}{\hbox{\epsfxsize=8.cm \epsfbox{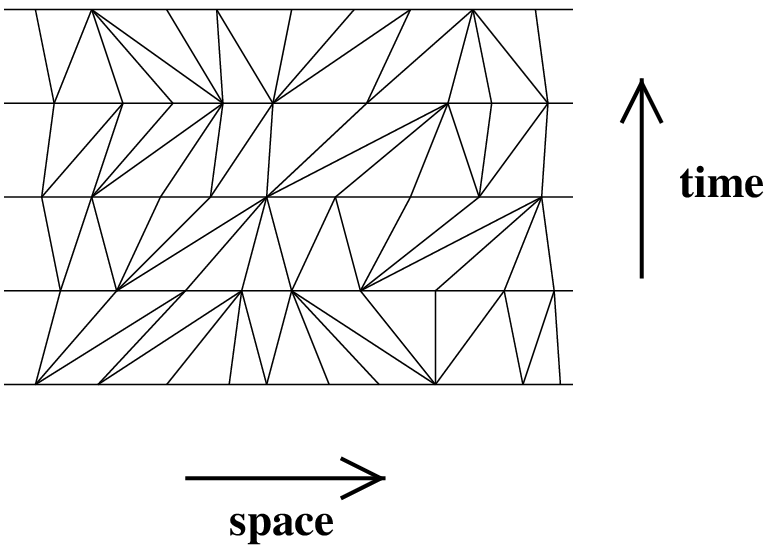}}}  $$

These triangulations are best described in the dual picture by considering triangles as vertical half-edges
and pairs of triangles that share a time-like (horizontal) edge as vertical edges between two consecutive time-slices.
We may now concentrate on the transition between two consecutive time-slices which typically reads as follows:
\begin{equation}\label{transmatbare} \raisebox{-1.cm}{\hbox{\epsfxsize=8.cm \epsfbox{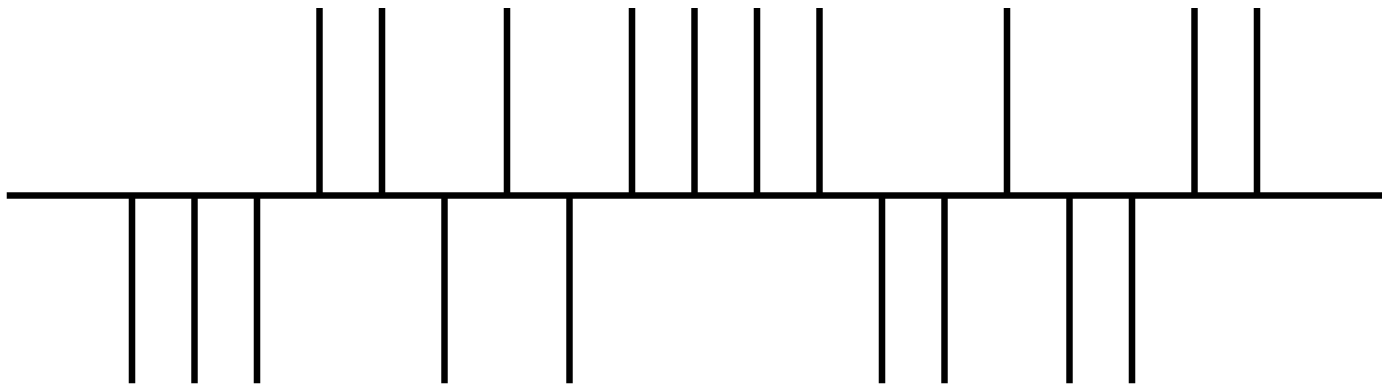}}}  \end{equation}
with say $i$ half-edges on the bottom and $j$ on the top (here for instance we have $i=9$ and $j=10$). 
Denoting by $\vert i\rangle$ and $\vert j\rangle$, with\footnote{Here and throughout the paper, 
we use the notation $\Z_+=\{0,1,2,...\}$.}
$i,j\in \Z_+$, the bottom and top Hilbert space state bases, we may describe the generation of a triangulation 
by the iterated action of a transfer operator ${\mathcal T}$ with matrix elements $T_{i,j}={i+j\choose i}$
between states $\vert i\rangle$ and $\vert j\rangle$.
Note that the corresponding matrix $T=(T_{i,j})_{i,j\in \Z_+}$ is infinite. 
We shall deal with such matrices in the following.
As detailed in Appendix \ref{genapp},
a compact characterization of the infinite matrix $T$ is via its double generating function:
$$ f_T(u,v)=\sum_{i,j\geq 0} T_{i,j} u^i v^j =\frac{1}{1-u-v} \, .$$

\subsection{Integrability}

To make the model more realistic, we may include both area and curvature-dependent terms, by introducing
Botlzmann weights $w(\Theta)$ equal to the product of local weights of the form $g$ per triangle (area term)
and $a$ per pair of consecutive triangles in a time-slice pointing in the same direction, either both up or both down
(curvature term).
The rules in the dual picture are as follows:
$$\raisebox{-1.cm}{\hbox{\epsfxsize=6.cm \epsfbox{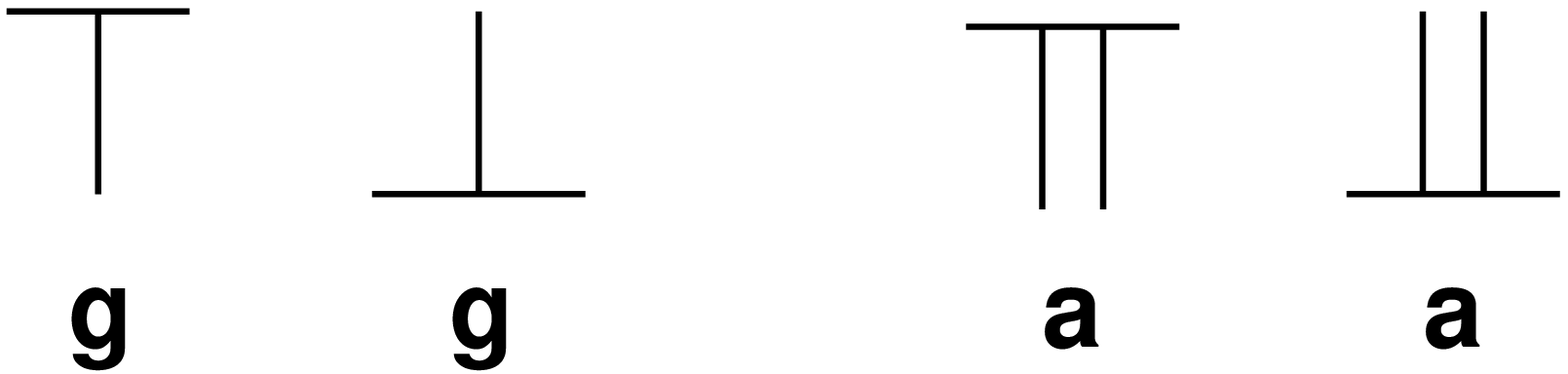}}}   $$
For instance, in the example \eqref{transmatbare} above with $i=9$ and $j=10$, 
the product of local weights is $g^{19} a^{9}$.
For the staircase-like boundary conditions of \cite{LORGRA},
namely assuming that each state-to-state transition as in \eqref{transmatbare}
has at least one leftmost half-edge on the bottom and one rightmost half-edge on top 
(and not counting the leftmost and rightmost half-edge weights $g^2$),
it is easy to compute the new transfer operator ${\mathcal T}(g,a)$ with
matrix elements 
\begin{equation}\label{Tga}
T(g,a)_{i,j}= (ag)^{i+j}\, \sum_{k=0}^{{\rm Min}(iij)} {i\choose k}{j\choose k} a^{-2k} 
\end{equation}
for $i,j\geq 0$ which expresses the transition between states $\vert i+1\rangle$ and $\vert j+1\rangle$.
Equivalently, the double generating function reads:
\begin{equation}\label{simple} f_{T(g,a)}(u,v)=\sum_{i,j\geq 0} T(g,a)_{i,j} u^i v^j =\frac{1}{1-ga(u+v)-g^2(1-a^2)u v} 
\end{equation}
and we have $T(g,a)=T(ga,ga,g^2(1-a^2))$ in the notations of Appendix \ref{genapp}.

This model turns out to provide one of the simplest examples of quantum integrable system, with an infinite
family of commuting transfer matrices. Indeed, we have:

\begin{thm}\cite{LORGRA}
The transfer matrices $T(g,a)$ and $T(g',a')$ commute if and only if the parameters $(g,a,g',a')$ are such that
$\varphi(g,a)=\varphi(g',a')$ where:
\begin{equation}\label{lorg} \varphi(g,a)=\frac{1-g^2(1-a^2)}{a g}
\end{equation}
\end{thm}

This is easily proved by using the generating functions. This case corresponds to the family of matrices
$T_{s,t}(\al)$ of Appendix \ref{genapp}, with $s=1$, $t=\varphi(g,a)$ and $\al=g a$.

\subsection{Diagonalization}\label{diagoT}

To diagonalize $T$, we may first consider finite size truncations $T^{[0,k]}=\big(T_{i,j}(g,a)\big)_{i,j\in [0,k]}$.
It is easy to see that the $(k+1)\times (k+1)$ symmetric matrix $T^{[0,k]}$ is diagonalizable, with eigenvalues
$\lambda_i^{[0,k]}(g,a)=g^{2i}(1+O(g^2))$, $i\in [0,k]$, all with formal series expansions in powers of $g$
with coefficients in $\Z[a]$. As $k$ increases, we get more and more eigenvalues, with power
series expansions that stabilize. In this sense, the limiting infinite matrix has an infinite set of eigenvalues
$\lambda_i=g^{2i}(1+O(g^2))$, $i\in\Z_+$, with well-defined formal power series expansions in $g$.

Setting $\varphi(g,a)=q+q^{-1}$ for some $q\in \C^*$, and introducing a new variable 
$$\lambda=\frac{1- q^{-1} g a}{1-q g a} \, ,$$
we may rewrite the double generating function of $T(g,a)$ as:
\begin{eqnarray}\label{gentga}
f_{T(g,a)}(u,v)&=&\frac{1- \lambda q^2}{(1-q u)(1-q v)-\lambda (u-q)(v-q)}\nonumber \\
&=&
\sum_{m=0}^\infty  \frac{\sqrt{1-q^2}(q-u)^m}{(1-q u)^{m+1}}
\left(\frac{1-\lambda q^2}{1-q^2}\lambda^m\right) \frac{\sqrt{1-q^2}(q-v)^m}{(1-q v)^{m+1}}
\end{eqnarray}
where we identify 
$$\Lambda^{(m)}=\frac{1-\lambda q^2}{1-q^2}\lambda^m$$
as the $m$-th eigenvalue of $T(g,a)$, $m\in\Z_+$ and
$$f_{v^{(m)}(u)}=\sum_{i=0}^\infty v_i^{(m)} u^i = \frac{\sqrt{1-q^2}(q-u)^m}{(1-q u)^{m+1}}$$
as the generating function for the corresponding eigenvector $v^{(m)}$. Note that $(v^{(m)})_{m\in \Z_+}$
form an orthonormal basis of the Hilbert space of states 
w.r.t. the standard scalar product $u\cdot v=\sum_{i\in\Z_+}u_i v_i$.
It is easy to show that $\Lambda^{(m)}=g^{2m}(1+O(g^2))$ as a formal power series of $g$, thereby proving that
these are the limits of the eigenvalues of the truncated matrices as the size $k\to \infty$.

This was extensively used in  \cite{LORGRA} to compute correlation
functions of top/bottom boundary loops in random Lorentzian triangulations. 
We want to stress here the very simple form of the generating function \eqref{simple}, 
which will reappear later in these notes.

\subsection{Trees}

For staircase boundary conditions, the dual random Lorentzian triangulations introduced above may
be viewed as random plane trees. This is easily realized by gluing all the bottom vertices of parallel vertical edges
whose both top and bottom halves contribute to the curvature term (no interlacing with the neighboring time slices).
A typical such example reads:
$$ \raisebox{-1.cm}{\hbox{\epsfxsize=10.cm \epsfbox{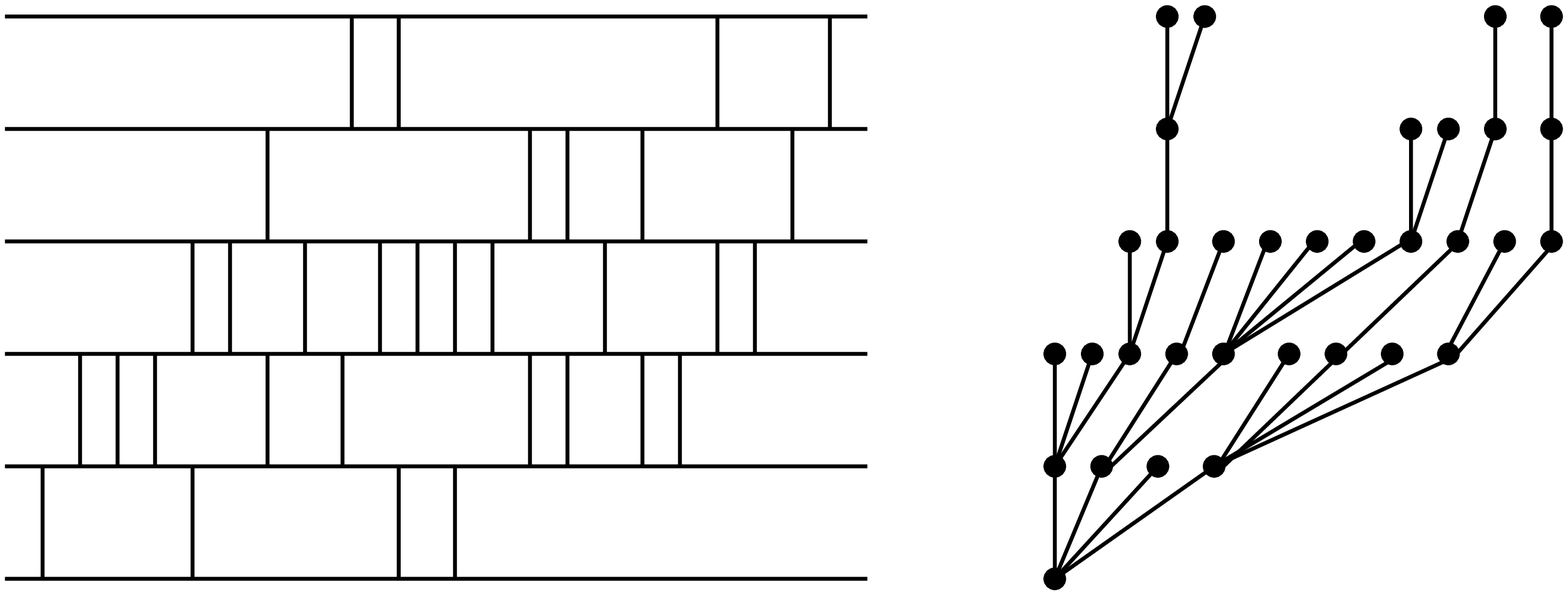}}}$$
Note that the tree is naturally rooted at its bottom vertex.

To summarize, we have unearthed some integrable structure attached naturally to certain plane trees. 
Note that in tree language the weights are respectively $g^2$ per edge 
(except the left/right boundary ones), and $a$ per pair of consecutive descendent half-edges (from left to right) and
per leaf.

\subsection{Paths}\label{pathLU}

There is yet another interpretation of the transfer matrix $T(g,a)$ of $1+1D$ Lorentzian gravity,
in terms of lattice paths. First notice that $T(g,a)=V(g,a)V^t(g,a)$ for some (infinite)
lower triangular matrix $V(g,a)$ with entries
\begin{equation}\label{lowerL}
V(g,a)_{i,k}= g^i a^{i-k} {i\choose k} \qquad (i,k \in \Z_+)
\end{equation}
and double generating function:
\begin{equation}\label{genL}
f_{V(g,a)}(u,v)= \frac{1}{1-a g u-g u v}
\end{equation}
In the notations of Appendix \ref{defapp}, we have $V(g,a)=L(a^{-1},ga)$.

\begin{figure}
\centering
\includegraphics[width=14.cm]{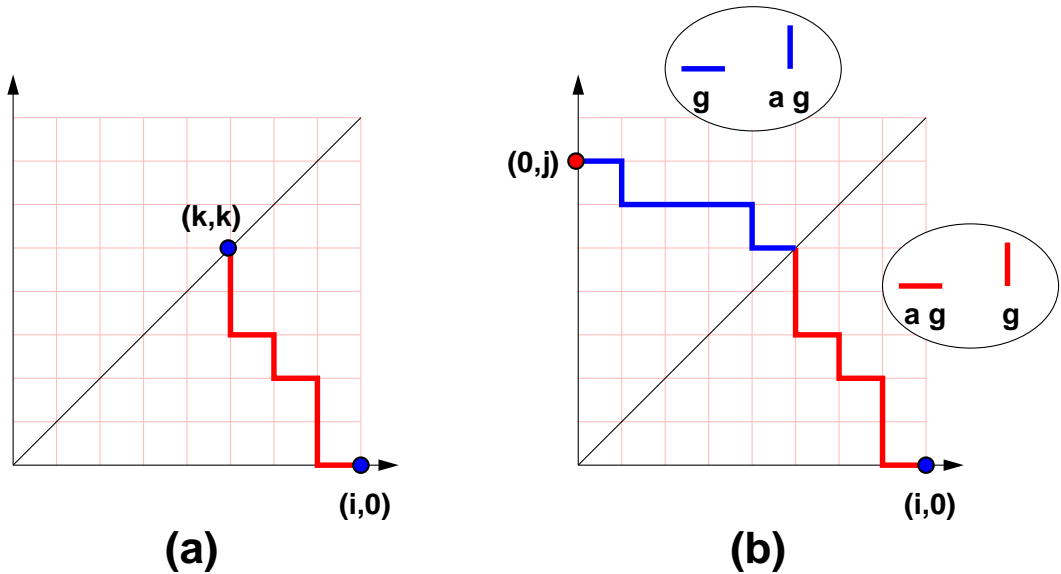}
\caption{\small A typical path contributing to the matrix element $V(g,a)_{i,k}$ (a), and to $T(g,a)_{i,j}$ (b).
These paths are taken on $\Z_+^2$, with steps $(-1,0)$ and $(0,1)$.
In the latter case, we have indicated the weights of the steps below and above the diagonal.}
\label{fig:tupaths}
\end{figure}

Consider paths on the positive quadrant of the two-dimensional square lattice $\Z+^2$, with steps
$(-1,0)$ (horizontal) and $(0,1)$ (vertical), as illustrated in Fig.\ref{fig:tupaths}(a). 
Then the total number of paths from the point $(i,0)$
to the point $(k,k)$ on the diagonal is ${i \choose k}$. Moreover if we attach a weight $g a$ per
horizontal step and $g$ per vertical one, we get a total contribution of 
$(g a)^{k-i} g^k{i \choose k}=V(g,a)_{i,k}$, which we interpret as the partition function for 
weighted paths
from $(i,0)$ to $(k,k)$. This quantity will also reappear later.

Note that in this language $T(g,a)_{i,j}$ is the partition
function of lattice paths from $(i,0)$ to $(0,j)$ in the positive quadrant, with weights $g a$ (resp. $g$) per
horizontal step below (resp. above) the diagonal and $g$ (resp. $g a$) per step above (resp. below)
the diagonal (see Fig.\ref{fig:tupaths}(b) for an illustration).

This formulation allows to visualize immediately the truncated transfer matrix $T^{[0,k]}(g,a)$ as 
corresponding to paths within the square $[0,k]\times[0,k]\subset\Z_+^2$. For such paths, both the
portion below the diagonal and that above are within the same square, so that we may write
$T^{[0,k]}(g,a)=V^{[0,k]}(g,a)V^{[0,k]}(g,a)^t$, wich immediately yields the determinant of $T^{[0,k]}$,
as $V^{[0,k]}(g,a)$ is lower triangular (see also Appendix \ref{detapp}):
$$ \det\big(T^{[0,k]}(g,a) \big)= \det\big(V^{[0,k]}(g,a) \big)^2=g^{k(k+1)} $$
This is compatible with eigenvalues $\lambda_i^{[0,k]}(g,a)=g^{2i}(1+O(g^2))$ for $i=0,1,...,k$.

\section{Enumerating ASMs}

\subsection{ASMs, 6V model and the IK deteminant}

As discovered by Kuperberg \cite{KUP}, ASMs of size $n\times n$ are in bijection with the so-called 
Domain-Wall Boundary Condition Six Vertex (6V-DWBC) model on a square grid of size $n\times n$.
The latter configurations are choices of orientations of the edges of a $n\times n$ grid of the 
(two-dimensional) square lattice, in such a way that at each vertex exactly two edges point to-
and two point from- the vertex. Moreover oriented external horizontal (resp. vertical) edges are 
attached to the boundary vertices, in such a way that external horizontal edges point towards the grid
and vertical ones from the grid. We display below the 6 possible vertex configurations $a_1,a_2,b_1,b_2,c_1,c_2$
obeying the above rules,  as well as a sample grid showing the external edge boundary condition:
$$\raisebox{-1.cm}{\hbox{\epsfxsize=14.cm \epsfbox{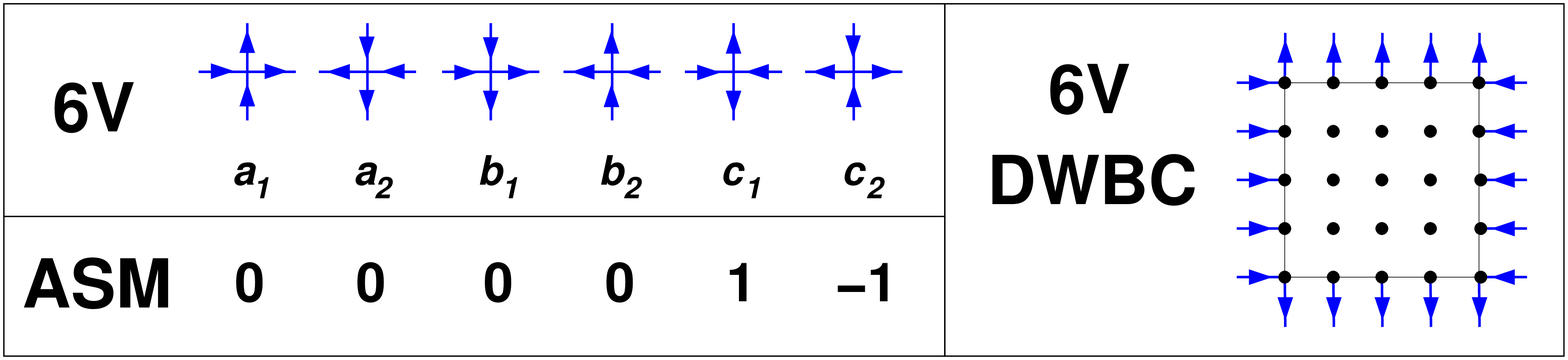}}} 
$$
We have also indicated the dictionary between the vertex configurations and the ASM entries. It is easy to understand
the bijection as follows. The $\pm 1$ entries correspond to vertices where both the horizontal and vertical flows
(indicated by the direction of the edges) are {\it reflected}. The alternation of $1$ and $-1$ entries corresponds 
to odd and even order flips along each row or column of the grid. The DWBC boundary condition
ensures that the first and last encountered non-zero elements in the ASM
along rows and columns must be $1$, as one needs an odd total number of flips to reverse the external edge
orientation, when going along a row or a column.

The 6V model had been extensively studied in the physics literature. With a suitable parameterization of the Boltzmann
weights $a_i,b_i,c_i$, the model forms the archetypical example of an integrable lattice model, as it admits an infinite family
of commuting transfer matrices, that can be diagonalized for various types of bopundary conditions using the Bethe Ansatz 
techniques.  These integrable weights are defined as follows.
Each row (resp. column)
of the grid carries a complex number $z_i$, $i=1,2,...n$ (resp. $w_i$, $i=1,2,...n$ ) called spectral parameter.
Moreover the weights depend on a ``quantum" parameter $q\in \C^*$.
We have the following parametrization of the weights:
$$ a(z,w)=q z-q^{-1} w \qquad b(z,w)=q^{-1} z -q w \qquad c(z,w)=(q^2-q^{-2}) \sqrt{z w} $$
where $a(z,w)$ is the weight for a vertex of type $a_1$ or $a_2$ at the intersection of a line with
parameter $z$ and column with parameter $w$, etc. With this parameterization, the model has an infinite
family of commuting row-to-row transfer matrices, and can be exactly solved by Bethe Ansatz techniques.
Using recursion relations of Korepin \cite{KO}, Izergin \cite{IZ} obtained a compact determinantal formula for the partition function
of this 6V-DWBC model, defined as the sum over edge configurations of the product of local vertex weights,
divided by the normalization factor $\prod_{i=1}^n c(z_i,w_i)$ (to make the answer polynomial in the $z$'s and $w$'s).
It reads:
\begin{equation}\label{IK} 
Z_{6V}^{(n)}(q;\{z_i\},\{w_j\})=\frac{\prod_{i,j} a(z_i,w_j)b(z_i,w_j)}{\Delta(z)\Delta(w)}
\det_{1\leq i,j\leq n} \left( \frac{1}{a(z_i,w_j)b(z_i,w_j)}\right)
\end{equation}
where $\Delta(z)=\prod_{1\leq i<j\leq n}(z_i-z_j)$ stands for the Vandermonde determinant of the $z$'s.

\subsection{Homogeneous limit and computation of $Z_{ASM}^{(n)}(x,y,1)$}

In the above bijection between 6V configurations and ASMs, it is easy to track both quantities 
$N(B)$ and ${\rm Inv}(B)$ in terms of 6V weights.
We find that 
$$N(B)= N_{c_2}=\frac{N_c-n}{2} \qquad {\rm Inv}(B)-N(B)=N_{a_1}=N_{a_2}=\frac{N_{a}}{2}$$
where $N_{a_i}$, $N_{b_i}$, $N_{c_i}$ stand for the total numbers of vertex configurations of each type.
The determinant result above can therefore be used to compute the refined partition function $Z_{ASM}^{(n)}(x,y,1)$
for ASMs, which counts ASMs with a weight $x/y$ per entry $-1$ and a weight $y$ for each inversion. Setting
\begin{equation}\label{xyab}
x=\left( \frac{c}{b}\right)^2\qquad y=\left( \frac{a}{b} \right)^2 
\end{equation}
we have:
\begin{equation}\label{zasm} 
Z_{ASM}^{(n)}(x,y,1)=\sum_{ASM\, B} x^{N(B)} y^{{\rm Inv}(B)-N(b)}=b^{-n(n-1)} \, Z_{6V}^{(n)}(a,b,c)
\end{equation}
where $Z_{6V}^{(n)}(a,b,c)$ refers to the homogeneous limit of the partition function \eqref{IK}
of the 6V model in which all $a(z_i,w_j)$
tend to a, etc. This is obtained by letting all $z_i\to r$ and all $w_i\to r^{-1}$, 
with $a=a(r,r^{-1})$, $b=b(r,r^{-1})$ and $c=c(r,r^{-1})$.
This and more refinements were worked out in \cite{BDZJ1}. We have the following remarkable result:

\begin{thm}\cite{BDZJ1}\label{asmthm}
The partition function for refined ASMs reads:
\begin{equation}\label{asmdet}
Z_{ASM}^{(n)}(x,y,1)= \det_{0\leq i,j \leq n-1} \left ( (1-\nu){\mathbb I}+\nu G\right) 
\end{equation}
where $\nu$ is any solution to the equation
\begin{equation}\label{nudef} x \nu(1-\nu)=\nu +y(1-\nu)\end{equation}
and the $n\times n$ determinant is the principal minor for the $n$ first rows and columns of the infinite
matrix $ M_{ASM}=(1-\nu){\mathbb I}+\nu G$ whose entries are generated by
\begin{equation}\label{fasm} f_{M_{ASM}}(u,v)=\frac{1-\nu}{1-u v}+\frac {\nu}{1-xu-v -(y-x)u v} 
\end{equation}
\end{thm}

Note the remarkable similarity between the generating function for the matrix elements of $G$ and the 
that of the transfer matrix for $1+1D$ Lorentzian triangulations \eqref{simple}. The two actually match up to a rescaling
$u\to u/\sqrt{x}$ and $v\to v\sqrt{x}$ (which amounts to a conjugation by the diagonal matrix $\sqrt{x}\, \mathbb I$)
and upon identifying $x=g^2a^2$ and $y=g^2$. Note also that $G=T(x,1,y-x)$ in the notations of Appendix \ref{defapp}.

Let us now give a sketch of the proof of Theorem \ref{asmthm}. The determinant \eqref{IK} is singular in the homogeneous
limit, but we may Taylor-expand the matrix entries within the determinant around the homogeneous point. For $a=a(r,r^{-1})$,
$b=b(r,r^{-1})$ and $c=c(r,r^{-1})$ this 
reads:
$$Z_{6V}^{(n)}(a,b,c)=\frac{(a b)^{n^2}}{c^n} \, \det_{0\leq i,j\leq n-1}\left(\left\{ \left(\frac{1}{i!}\frac{d^i}{du^i}\right)
\left(\frac{1}{j!}\frac{d^j}{dv^j}\right)
\frac{c(u^{-1},v)}{a(u^{-1},v)b(u^{-1},v)}\right\}\Big\vert_{u=v=r^{-1}} \right) $$
Noting further that
$$ \frac{c(u^{-1},v)}{a(u^{-1},v)b(u^{-1},v)} = \frac{1}{u v-q^{-2}} -\frac{1}{u v -q^{2}} $$
and introducing the infinite matrices $A_{\pm}$ with elements:
\begin{equation}\label{matelem} (A_{\pm})_{i,j}=\left\{ \left(\frac{1}{i!}\frac{d^i}{du^i}\right)
\left(\frac{1}{j!}\frac{d^j}{dv^j}\right)
 \frac{1}{u v-q^{\pm 2}}\right\}\Big\vert_{u=v=r^{-1}} \qquad (i,j\in \Z_+)
\end{equation}
we have the following straightforward:
\begin{lemma}\label{apdecomp}
We have
\begin{equation}\label{taylorapm}
A_\pm=\frac{1}{r^{-2}-q^{\pm 2}}\, \left( U(\al_\pm,\beta_\pm)^t \right)^{-1} \, U(\al_\pm',\beta_\pm') 
\end{equation}
for $U(\al,\beta)$ the infinite upper triangular matrix with entries $U(\al,\beta)_{i,j}={j\choose i}\al^i\beta^j$ 
(see Appendix \ref{defapp}),
and where the parameters read: 
$$ \al_+ =\frac{1-q^2r^2}{r}\, , \quad \beta_+=\frac{q^2-q^{-2}}{r^2-q^2} \, , \quad \al_+'=-q^2r^2\beta_+\, ,\quad
\beta_+'=-\frac{1}{\al+} $$
and the parameters with $-$ index are obtained form those with $+$ by the substitution $q\to q^{-1}$.
\end{lemma}
\begin{proof} The statement of the lemma is an immediate consequence of the fact that the Taylor-expansion expression
\eqref{matelem} around $(u,v)=(r^{-1},r^{-1})$ turns into the following double generating functions for the matrix elements
of $A_\pm$:
$$ f_{A_\pm}(u,v)=\sum_{i,j\in \Z_+} (A_\pm)_{i,j}u^i v^j= \frac{1}{(r^{-1}+u)(r^{-1}+v)-q^{\pm 2}} $$
Moreover, using the generating function for the matrix elements of $U(\al,\beta)$ of Appendix \ref{defapp}:
$$f_{U(\al,\beta)}(u,v)=\frac{1}{1-\beta v(1+\al u)} \, ,$$
and for $\al\beta\neq 0$, $U(\al,\beta)^{-1}=U(-1/\beta,-1/\al)$, we finally compute by convolution product:
\begin{eqnarray*}f_{U^t(\al,\beta)^{-1}U(\al',\beta')}(u,v)&=& f_{U(-1/\beta,-1/\al)}(v,u)*f_{U(\al',\beta')}(u,v) \\
&=&\oint_{\mathcal C} \frac{dt}{2i\pi t}\,  \frac{1}{1+\frac{1}{\al} u(1-\frac{1}{\beta} t^{-1})} \, \frac{1}{1-\beta' v(1+\al' t)}\\
&=&\frac{1}{(1+\frac{1}{\al}u)(1-\beta'v)-u v\frac{\al'\beta'}{\al\beta}}
\end{eqnarray*}
The lemma follows from comparing this with the expressions for $f_{A_\pm}(u,v)$.
\end{proof}

The Lemma is easily extended to finite truncations of the infinite matrices $A_\pm$ and $U$, as $U$ is upper triangular and
$(U^t)^{-1}$ is lower triangular. Therefore we may write 
$A_\pm^{[0,n-1]}=\frac{1}{r^{-2}-q^{\pm 2}} (U_\pm^{[0,n-1]\ t})^{-1}{U_\pm^{[0,n-1]}}'$ 
with the obvious shorthand notations. Going back to our original determinant, we find that
\begin{eqnarray*}&&Z_{6V}^{(n)}(a,b,c)=\frac{(a b)^{n^2}}{c^n} \, \det\left( A_-^{[0,n-1]}-A_+^{[0,n-1]}\right) \\
&=&
\frac{(a b)^{n^2}}{c^n} \det(A_-^{[0,n-1]})
\det\left({\mathbb I}- \frac{r^{-2}-q^{-2}}{r^{-2}-q^{2}} U_-^{[0,n-1]\ t} 
(U_+^{[0,n-1]\ t})^{-1}{U_+^{[0,n-1]}}'\big({U_-^{[0,n-1]}}'\big)^{-1} \right) 
\end{eqnarray*}
where $\mathbb I$ stands for the $n\times n$ identity matrix.
The last product of 4 matrices is finally identified with the matrix $G$ of Theorem \ref{asmthm} and computed by
use of generating functions, whereas the determinant of $A_-^{[0,n-1]}$ follows from its expression as a product
of triangular matrices. Collecting all the factors finally yields (\ref{asmdet}-\ref{fasm}), with the following identification
of parameters:
$$ x=\left(\frac{q^2-q^{-2}}{q^{-1}r-q r^{-1}}\right)^2\, , \quad y=\left(\frac{q r-q^{-1}r^{-1}}{q^{-1}r-q r^{-1}}\right)^2\, , \quad
\nu=\frac{r^{-2}-q^{-2}}{q^2-q^{-2}} \, , \quad 1-\nu=\frac{q^2-r^{-2}}{q^2-q^{-2}} \, .$$

\section{Enumerating DPPs}

\subsection{Lattice path formulation of DPPs}

\begin{figure}
\centering
\includegraphics[width=13.cm]{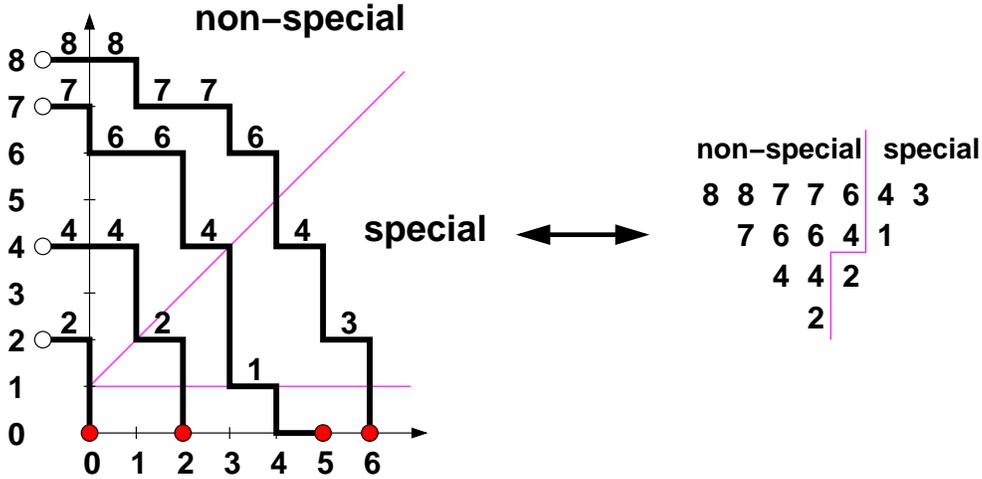}
\caption{\small Non-intersecting lattice path configuration for a sample DPP of order $n\geq 8$.
We have indicated the domains in which horizontal steps correspond to special/non-special parts.}
\label{fig:dppnilp}
\end{figure}

The DPPs are in bijection with configurations of non-intersecting lattice paths illustrated in Fig.\ref{fig:dppnilp}
and defined as follows. Like in Sect. \ref{pathLU}, the paths take place in the positive quadrant
$(\Z_+)^2$, with the same steps but different weights and boundary conditions.
The paths start along the $x$ axis at positions of the form $(s_i,0)$ ($i=1,2,...,r$ recorded from right to left) 
and end along the $y$ axis at positions $(0,s_i+2)$ ($i=1,2,...,r$ recorded from top to bottom). 
We add a final horizontal left step at the end of each path. Reading paths from left to right 
and top to bottom, we record the vertical positions $y=a_{i,j}$ of the $j$-th horizontal step from 
the left taken on the $i$-th path from top (steps with $y=0$ are not recorded). 
These form a DPP with $r$ rows, of order any $n\geq s_1+2$. 
Conversely to each DPP with $r$ rows we may associate such
a path configuration. Note that the starting points are such that $s_i=\lambda_i-1$, where 
$\lambda_i=\mu_i-i+1$ the total number of parts in the row $i$. 

The special parts correspond to horizontal steps 
taken in the strict upper octant $y\geq x+1$ of the plane, and the remaining parts correspond 
to the horizontal steps in the domain $1\leq y\leq x+1$,
while horizontal steps along the x axis do not count (weight $1$).

\subsection{Computation of $Z_{DPP}^{(n)(}(x,y,1)$}\label{compudpp}

The computation of $Z_{DPP}^{(n)}(x,y,1)$ uses the Lindstr\"om-Gessel-Viennot \cite{LGV1,LGV2} determinant formula
expressing the partition function for non-intersecting lattice paths with fixed atarting points and endpoints
$Z$ as a determinant $\det(Z_{i,j})$ where $Z_{i,j}$ is the partition function for a single path from the $i$-th starting point to
the $j$-th endpoint. This leads to the following:

\begin{thm}\cite{BDZJ1}\label{dppthm}
The partition function $Z_{DPP}^{(n)}(x,y)$ for DPPs of order $n$
with weight $x$ per special part and $y$ per other part reads:
\begin{equation}\label{dppdet}
Z_{DPP}^{(n)}(x,y)= \det \left( {\mathbb I} + H^{[0,n-1]}\right)
\end{equation}
where the determinant is that of the finite truncation to the $n$ first rows and columns 
of the infinite matrix $(M_{DPP})_{i,j}=\delta{i,j}+H_{i,j}$, $i,j\in \Z_+$, 
with generating function:
\begin{equation}\label{fdpp} f_{M_{DPP}}(u,v)=\sum_{i,j\in\Z_+} (M_{DPP})_{i,j}u^iv^j=
 \frac{1}{1-u v}+\frac{1}{1-u}\, \frac{y u}{1-x u-v -(y-x) u v}
\end{equation}
\end{thm}

Again, note the close similarity between the matrix $H$ and the transfer matrix $T$ for 
$1+1D$ Lorentzian triangulations. Our proof of the ASM-DPP conjecture will be based on this similarity.
Note also that in the notations of Appendix \ref{defapp}, we have $H=y S({\mathbb I}-S)^{-1}T(x,1,y-x)$.

Let us now sketch the proof of Theorem \ref{dppthm}. The sought after partition function
is a sum over all configurations of $r$ non-intersecting paths ($0\leq r\leq n-1$) with fixed $r$
starting points $(s_i,0)$, $i=1,...,r$ and endpoints $(0,s_i+2)$, $i=1,...,r$. 
According to the Lindstr\"om-Gessel-Viennot theorem, this is the sum over minors
$|D|_{s_1,...,s_r}^{s_1,...,s_r}$ of the $n\times n$ matrix $D$
whose entries $D_{i,j}$ is the partition function for a single path from $(i,0)$ to $(0,j+2)$.
It also has the simple expression:
$$\sum_{r=0}^{n-1}\, \sum_{0\leq s_1<...<s_r\leq n-1}|D|_{s_1,...,s_r}^{s_1,...,s_r} =\det({\mathbb I}+D)$$

Such a path is split into three pieces: (i) between the $x=0$ axis and the first hit on the $x=1$ axis 
(ii) between the $x=1$ axis and the diagonal line $y=x+1$ 
(iii) between the diagonal line $y=x+1$ and the vertical axis $y=0$. Each piece
receives a specific weight, with total contribution:
\begin{equation}
\label{Delem}
D_{i,j}=\sum_{k=0}^{i} \sum_{\ell=0}^{Min(k,j+1)}{k\choose \ell} x^{k-\ell} {j+1\choose \ell} y^{\ell+1} 
\end{equation}
where we have first summed over (i) paths from $(i,0)$to $(k,1)$ with $i-k$ horizontal steps along the $x=0$ axis
and one final vertical step (ii) paths from $(k,1)$ to $(\ell,\ell+1)$ on the diagonal $y=x+1$, for which $k-\ell$ horizontal
steps must be chosen among a total of $k$, each weighted by $x$ as these correspond to special parts (iiii) paths
from $(\ell,\ell+1)$ to $(0,j+2)$ for which $\ell$ horizontal steps must be chosen among a total $j+1$, each 
weighted by $y$ as they correspond to non-special parts with one extra $y$ factor for the additional final horizontal step.

Note that the matrix elements of $D$ are independent of $n$. We may therefore consider the extension of $D$
to an infinite matrix $\tilde D$ with matrix elements given by $D_{i,j}$ of eqn.\eqref{Delem}, for $i,j\in \Z_+$.
The theorem follows by identifying the infinite matrix $H$ with $\tilde D$, and therefore $H^{[0,n-1]}$ with $D$.

\section{Proof of the ASM-DPP conjecture}

\subsection{Proof of $Z_{ASM}^{(n)}(x,y,1)=Z_{DPP}^{(n)}(x,y,1)$ }

The expressions \eqref{asmdet} and \eqref{dppdet} for respectively the partition functions 
$Z_{ASM}^{(n)}(x,y,1)$ and $Z_{DPP}^{(n)}(x,y,1)$ are determinants of the principal minor of 
size $n$ of some infinite matrix, in other words, these are
the determinants of a finite truncation to the $n$ first rows and columns of infinite matrices. 

There is a very simple relation (independent of $n$)
between the generating functions of the two infinite matrices $M_{ASM}$ and $M_{DPP}$, namely:
\begin{equation}\label{identfin}
(1-\frac{u}{1-\nu})(1-v)f_{M_{ASM}} (u,v)=(1-u)(1-(1-\nu)v) f_{M_{DPP}}(u,v)
\end{equation}
as a direct consequence of \eqref{nudef}.

Let us translate this back into a finite matrix relation upon truncation.
First, for any matrix $A$ with generating function $f_A(u,v)$ the function $f_M(u,v)=(1-a u)(1-b v)f_A(u,v)$ is actually
the generating function of the infinite matrix $M=(\mathbb I- a S)A(\mathbb I -b S^t)$ where $S$ is the strictly lower triangular
shift matrix with elements $S_{i,j}=\delta_{i-j,1}$ for $i,j\in \Z_+$, and $S^t$ its strictly upper triangular transpose.
Upon truncation to indices in $[0,n-1]$, we have the obvious relation (see Lemma \ref{unitri} in Appendix \ref{genapp}): 
$M^{[0,n-1]}=(\mathbb I- a S)^{[0,n-1]} A^{[0,n-1]} (\mathbb I -b S^t)^{[0,n-1]}$, due to lower triangularity of $\mathbb I- a S$
and upper triangularity of $\mathbb I -b S^t$. Note that both matrix truncations are unitriangular, hence have determinant $1$ 
so that $\det(M^{[0,n-1]})=\det(A^{[0,n-1]})$ for all $n\geq 1$.
By the identity \eqref{identfin}, we therefore conclude that the truncations $M_{ASM}^{[0,n-1]}$ and $M_{DPP}^{[0,n-1]}$
have the same determinant, and the $z=1$ version of
Theorem \ref{mrrthm} follows.

\subsection{Refinement: proof of the MRR conjecture}

The observable $t(B)$ for ASMs $B$ may be included by slightly modifying the homogeneous limit
of the IK determinant. We simply have to consider vertex weights with homogeneous limits in
$\{a(r,r^{-1}),b(r,r^{-1}),c(r,r^{-1})\}$ at points $(i,j)$, $i=1,2,...,n$ and $j=1,2,...,n-1$ of the square grid,
and different weights $\{a(s,s^{-1}),b(s,s^{-1}),c(s,s^{-1})\}$ for the last column $i=1,2,...,n$ and $j=n$.
Defining further 
$$ z=\frac{a(s,s^{-1}) b(r,r^{-1})}{b(s,s^{-1}) a(r,r^{-1})} $$
gives an extra contribution $z^{t(B)}$ to the ASM enumeration.

Adapting the method of enumeration described above, one finds that we simply have to
change the definition of the last column of $M_{ASM}^{[0,n-1]}$ to include the $z$ dependence.
This in turn is obtained by modifying {\it all} columns of index $j\geq n-1$ in the infinite matrix
$M_{ASM}$ (we refer the reader to \cite{BDZJ1} for the technical details).  The result is the following:

\begin{thm}
The quantity $(1+\nu(z-1))Z_{ASM}(x,y,z)$ is the determinant of the truncation to the $n$ first rows
and columns of the modified infinite matrix $M_{ASM}'$, with double generating function
\begin{eqnarray*}
f_{M_{ASM}'}(u,v)&=&\frac{1-\nu}{1-u v} +\frac {\nu}{1- x u-v -(y-x)uv} \\
&&+\frac{\nu (z-1)}{1-(y(z-1)+x)u}\left( 1+\frac{y u}{1-x u}\right)^n v^{n-1} \left(
1+\frac{v}{x} \frac{y(\nu-1)+\nu(x u-1)}{\nu+(y-\nu x)u}\right)\end{eqnarray*}
\end{thm}

The prefactor $(1+\nu(z-1))$ is ad-hoc and comes from a modification of the columns $n$ and higher
in the infinite matrix to make it simpler.
Note that the new infinite matrix $M_{ASM}'$ has an explicit dependence on $n$.

Likewise, keeping track of the observable $M(A)$ in a DPP $A$ is easy.
The lattice path formulation still holds and yields a LGV-like determinant as well,
but for a modified $n\times n$ matrix $D_{i,j}'$, identical to $D_{i,j}$ of \eqref{Delem} for
$i=1,2,...,n$ and $j=1,2,...,n-1$ and with a different last column, explicitly depending on $z$.

The latter dependence is the result of
decomposing further the piece (iii) of the DPP (see Sect.\ref{compudpp}), when $j=n$,
into $(iii-a)$ and $(iii-b)$ and attaching weights as
follows:
(iii-a) the piece of the path from $(\ell,\ell+1)$ to its first vertex on the $y=n$ line at $(m,n)$
with a total of ${n-m-1\choose \ell-m}$ paths all with weight $y^{\ell-m}$
and (iii-b) the straight path from $(m,n)$ to $(0,n)$ with an extra weight of $(y z)^{m+1}$.
This gives:
\begin{equation}
D_{i,j}'=\sum_{k=0}^{i} \sum_{\ell=0}^{k}{k\choose \ell} x^{k-\ell}  {n-m-1\choose \ell-m} y^{\ell+1} z^{m+1}
\end{equation}
This leads to the following:

\begin{thm}
The quantity $(1+\nu(z-1))Z_{DPP}(x,y,z)$ is the determinant of the truncation to the $n$ first rows
and columns of the modified infinite matrix $M_{ASM}'$, with double generating function:
\begin{eqnarray*}
f_{M_{DPP}'}(u,v)&=& \frac{1}{1-z w}+\frac{1}{1-z}\frac{y z}{1-x z-w -(y-x) z w} \\
&&+(z-1) \frac{1-v}{1-u}\frac{yu+\nu(1-x u)}{1-(y(z-1)+x)u} \left( 1+\frac{y u}{1-x u}\right)^n v^{n-1}
\end{eqnarray*}
\end{thm}

Like $M_{ASM}'$, the infinite matrix $M_{DPP}'$ has an explicit dependence on $n$.

The proof of the complete Mills-Robbins-Rumsey conjecture follows from the following elementary lemma,
easily proved by direct computation:
\begin{lemma}
We have the relation:
$$ (1+(y-x \nu -1)u)(1-v)f_{M_{ASM}'}(u,v)=(1-u)(1+(\nu-1)v)f_{M_{DPP}'}(u,v)$$
\end{lemma}

As explained above, such a relation between the two infinite matrices $M_{ASM}'$ and $M_{DPP}'$
guarantees that the determinant of their truncation to their first $m$ rows and columns
coincide, for any $m\geq 1$, so it holds in particular for $m=n$ and Theorem \ref{mrrthm} follows.

\subsection{More refinements}

In \cite{BDZJ2} a further observable was considered for ASMs and DPPs. For any ASM $B$ let $b(B)$ 
be the number of $0$ entries to the right of the unique $1$ in the bottom row of $B$.
For any DPP $A$ of order $n$, let $P(A)$ be the number of parts equal to $n-1$ plus the number of rows 
of length $n-1$. Defining the two following partition functions:
\begin{eqnarray*}Z_{ASM}^{(n)}(x,y,z,w)&=&\sum_{n\times n\, ASM\, B} x^{N(B)}y^{Inv(B)-N(B)}z^{t(B)}w^{b(B)}\\
Z_{DPP}^{(n)}(x,y,z,w)&=&\sum_{DPP\, A\, {\rm of} \, {\rm order}\, n} x^{S(A)}y^{NS(A)}z^{M(A)}w^{P(A)}
\end{eqnarray*}
for respectively ASMs of size $n$ and DPPs of order $n$, we have:
\begin{thm}\cite{BDZJ2}
We have the identity:
$$ Z_{ASM}^{(n)}(x,y,z,w)=Z_{DPP}^{(n)}(x,y,z,w) $$
\end{thm}

This was proved in \cite{BDZJ2} by showing that both functions
$Z_{ASM}^{(n)}(x,y,z,w)$ and $Z_{DPP}^{(n)}(x,y,z,w)$ obey the following relation as functions of $z,w,n$:
$$ (z-w)Z^{(n)}(z,w)Z^{(n-1)}(1,1)=(z-1)w Z^{(n)}(z,1)Z^{(n-1)}(1,w)-(w-1)zZ^{(n-1)}(z,1)Z^{(n)}(1,w) $$
in both cases as a consequence of the Desnanot-Jacobi identity \eqref{desnajac}.

\section{Conclusion}

\subsection{ASM,DPP,TSSCPP, FPL,DPL, etc.}

In these notes, we have detailed the refined enumeration of ASMs and DPPs and established an
identity between them. 

\begin{figure}
\centering
\includegraphics[width=15.cm]{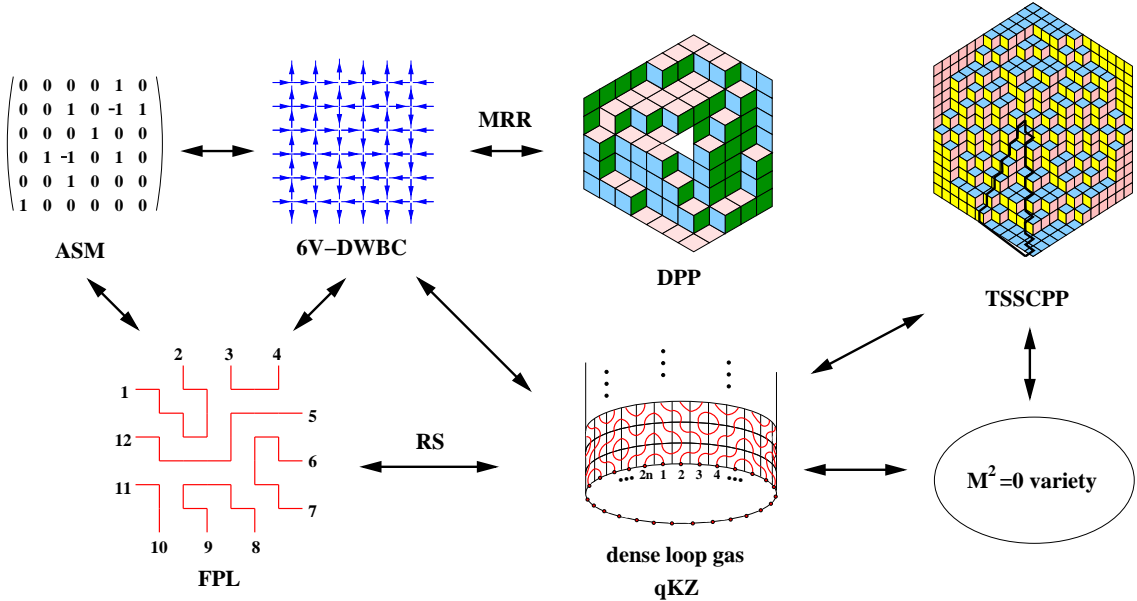}
\caption{\small From left to right: ASM, 6V-DWBC and FPL, all in bijection;
dense loop gas (DPL): its groundstate/limiting probability vector satisfies the qKZ equation, the components measure 
FPL correlations (RS conjecture), their sum matches the 6V-DWBC partition function with inhomogeneous 
spectral parameters $z_i,w_j$ and $q^3=1$; DPPs: their refined evaluation matches that of ASMs (MRR conjecture);
TSSCPPs: their refined enumeration matches a sum rule for qKZ solutions at generic $q$ and $z_i=1$; Variety $M^2=0$: 
its degree/multidegree matches solutions of qKZ for $q=1$.}
\label{fig:everything}
\end{figure}

One could think of further refinements, leading eventually to a bijection between
these objects. This is however only the tip of a much larger iceberg (see Fig.\ref{fig:everything}), 
which on the pure combinatorics
side involves other objects: the so-called TSSCPPs (Totally Symmetric Self-Complementary Plane Partitions)
which are yet another kind of plane partitions, with a formulation in terms of different configurations
of non-intersecting lattice paths (see for instance \cite{Bressoud} for a detailed account). 
There is also a statistical physics side, involving the so-called Fully Packed Loop
(FPL) model on a square grid, whose configurations are in bijection with those of the 6V-DWBC model.
The latter plays a central role in the so-called Razumov-Stroganov (RS) conjecture \cite{RS} proved
by Cantini-Sportiello \cite{CS}, 
relating its refined enumeration
according to link patterns of connections of the loops around the grid
to the asymptotic probabilities of connections of the 
Densely Packed Loop (DPL) model on a semi-infinite cylinder of finite perimeter. The latter is yet another
integrable lattice model based on some pictorial representation of the Temperley-Lieb algebra, whose groundstate
vector is a solution to the  quantum Knizhnik-Zamolodchikov (qKZ) equation \cite{DZJ0,DZJ1}.
Finally, there is an algebraic geometry side of the iceberg. For instance, the degree of the variety of upper triangular complex matrices with vanishing square corresponds to a refined enumeration of TSSCPPs, and the (equivariant cohomology)
multidegree is obtained via a specialization of the solution to the qKZ equation \cite{DZJ2}.

Many of the known enumerations of the above objects involve determinants,. In a number of cases,
these can be obtained through some application of the Lindstr\"om-Gessel-Viennot theorem.
This applies to all the ``free fermion" cases that are
in bijection with non-intersecting lattice path configurations. However, both the 6V model and the DPL model
are models of interacting fermions, in which even if there is some kind of lattice path formulation, the latter are
no longer just non-intersecting. For instance, in the case of the 6V model, one can define paths going
from the left border of the grid to the top border, by going right and up along the oriented edges
as much as possible (i.e. when there is a choice, always go up). Such paths are now ``osculating" in that
two paths can bounce against each other at a vertex (the first going right, then up; the second going up then right;
this corresponds to the vertex $a_1$ of the 6V model),
and this configuration receives a different weight, interpreted as the exponential of some interaction energy.
Yet, somehow, our formula for the refined enumeration has magically disentangled this interaction, to
make it look like a free fermion model, via our determinant evaluation. This
mechanism deserves to be better understood.

\subsection{Integrabilities}

The 6V model is the archetype of 2D integrable lattice model, related to the 1D quantum XXZ spin chain
for $sl_2$. It is know to have an infinite family of transfer matrices $T(a,b,c)$ provided
the Boltzmann weights $a,b,c$ satisfy the following relation:
$$\Delta(a,b,c)=\frac{a^2+b^2-c^2}{2 a b}={\rm const.}$$
This constant is the anisotropy of the associated quantum spin chain.
Alternatively, in terms of the $x,y$ variables of \eqref{xyab}, this turns into the following ``6V" variety:
\begin{equation}\label{xysixv}\psi(x,y)=\frac{1+y-x}{\sqrt{y}} ={\rm const.} 
\end{equation}
as $x,y>0$.

On the other hand, the infinite matrix $M_{ASM}$, whose finitely truncated determinant gives the DWBC
homogeneous 6V partition function on a grid of same size, also involves a transfer matrix $\theta$ of
a form analogous to that of 1+1D Lorentzian triangulations, generated by:
$$f_\theta(u,v)=\frac{1}{1-x u-v -(y-x) u v}$$
The transfer matrices $\theta$ commute for different values of the parameters $x,y$ provided
they belong to the following ``Lorentzian" variety:
\begin{equation}\label{xylor}\varphi(x,y)=\frac{1+x-y}{\sqrt{x}}={\rm const.} 
\end{equation}
obtained by rephrasing \eqref{lorg} above.

Comparing \eqref{xysixv} and \eqref{xylor}, we see that $\varphi(x,y)=\psi(y,x)$, hence the 
two varieties are distinct! However, they do intersect. Solving for
$\varphi(x,y)=q+q^{-1}$ and $\psi(x,y)=p+p^{-1}$ with say $q,p>1$, we find that
$$\sqrt{ x}=\frac{p(q^2-1)}{p^2q^2-1} \qquad \sqrt{y}=\frac{q(p^2-1)}{p^2q^2-1} $$
Conversely, any such point for $p,q>1$ lies at the intersection of two ``integrable" varieties
of the form \eqref{xysixv} and \eqref{xylor}.
This intriguing fact deserves a better understanding. In particular, the 6V variety 
involves commutation of {\it finite} size transfer matrices, whereas the Lorentzian one
concern matrices of infinite size.

\begin{appendix}
\section{Infinite matrices and truncated determinants}\label{genapp}
Throughout these notes, we make extensive use of generating functions for infinite
matrices. Let us summarize here the main definitions and properties we use.

\subsection{Infinite matrices}

We consider infinite matrices $A=(a_{i,j})_{i,j\in \Z_+}$.
The very concept of an infinite matrix is a bit delicate to work with, for instance
the product of two such matrices might not be well defined. This may be repaired
by introducing a formal expansion parameter $\epsilon$, and associating to $A$
the matrix $A(\epsilon)=(\epsilon^{i+j} a_{i,j})_{i,j\in \Z_+}$. The product of any two such
matrices now makes sense in the sense of formal power series of $\epsilon$.
Moreover, even the notion of eigenvector and eigenvalue make sense in this setting,
provided one can show that the latter have formal power or Laurent series expansions
in $\epsilon$.

Such a construction will always be implicit (when not explicit) throughout these notes.
For instance, the parameter $g$ in $T(g,a)$ of (\ref{Tga}-\ref{simple}) plays the role of $\epsilon$.
The ``diagonalization" of $T(g,a)$ along the integrable variety $\varphi(g,a)=q+q^{-1}$ 
(see Sect.\ref{diagoT}) is an example
of such extended notions of eigenvectors and eigenvalues.

\subsection{Generating functions for infinite matrices}\label{defapp}

For an infinite matrix $A=(a_{i,j})_{i,j\in\Z_+}$ and a vector $w=(w_i)_{i\in\Z_+}$ we define the
formal generating functions
$$ f_A(u,v)=\sum_{i,j\in\Z_+} a_{i,j}u^iv^j \qquad f_w(u)=\sum_{i\in \Z_+} w_i u^i  $$
with the following properties for matrices $A,B$ and a vector $w$:
\begin{eqnarray*}
f_{\mathbb I}(u,v)&=&\frac{1}{1-u v}\\
f_{A^t}(u,v)&=&f_A(v,u)\\
f_{Aw}(u)&=&\oint\frac{d t}{2i\pi t} f_A(u,t^{-1})f_w(t) \\
f_{AB}(uv)&=&(f_A*f_B)(u,v)=\oint\frac{d t}{2i\pi t}f_A(u,t^{-1})f_B(t,v)\\
\end{eqnarray*}
where the contour integral picks the constant term in $t$.

We consider the lower and upper triangular matrices $L(\al,\beta)$ and $U(\al,\beta)$ with generating
functions
\begin{equation}\label{defLU} 
f_{L(\al,\beta)}(u,v)=\frac{1}{1-\beta u(1+\al v)} \qquad f_{U(\al,\beta)}(u,v)=\frac{1}{1-\beta v(1+\al u)}
\end{equation}
with $U(\al,\beta)=L(\al,\beta)^t$.
Let us also introduce the shift matrix $S=(\delta_{i,j+1})_{i,j\in\Z_+}$ and the transfer matrix 
$T(\al,\beta,\gamma)$ generated respectively by:
\begin{equation}\label{defST} f_S(u,v)=\frac{u}{1-u v}\qquad  
f_{T(\al,\beta,\gamma)}(u,v)=\frac{1}{1-\al u -\beta v -\gamma u v}\end{equation}
In particular, we have
$$ L(\al,\beta)=T(\beta,0,\al\beta)\qquad U(\al,\beta)=T(0,\beta,\al\beta)\qquad {\mathbb I}=T(0,0,1)$$
The case of the infinite transfer matrix $T(g,a)$ for Lorentzian triangulations corresponds to the identification: 
$$ T(g,a)= T(ga,ga,g^2(1-a^2)$$

We have the following properties easily derived by contour integrals for the corresponding generating functions:
\begin{eqnarray}
L(\al,\beta)L(\al',\beta')&=&L\left(\frac{\al\beta'\al'}{1+\al\beta'},\beta(1+\al\beta')\right)\label{propone}\\
U(\al,\beta)U(\al',\beta')&=&U\left(\frac{\al\beta\al'}{1+\al'\beta},\beta(1+\al'\beta)\right)\label{proptwo}\\
L(\al,\beta)^{-1}=L\left(-\frac{1}{\beta},-\frac{1}{\al}\right)\ \ && \ \  U(\al,\beta)^{-1}
=U\left(-\frac{1}{\beta},-\frac{1}{\al}\right)\label{propthree}\\
L(\al,\beta)U(\al',\beta')&=&T\left(\beta,\beta',\beta\beta'(\al\al'-1)\right)\label{propfour}\\
U(\al',\beta')L(\al,\beta)&=&\frac{1}{1-\beta\beta'}\,
T\left(\frac{\al'\beta\beta'}{1-\beta\beta'},\frac{\al\beta\beta'}{1-\beta\beta'},
\frac{\al\al'\beta\beta'}{1-\beta\beta'}\right)\label{propfive}\\
T(\al,\beta,\gamma)T(\al',\beta',\gamma')&=&\frac{1}{1-\beta\al'} \,
T\left(\frac{\al+\gamma\al'}{1-\beta\al'},\frac{\beta'+\gamma'\beta}{1-\beta\al'},
\frac{\gamma\gamma'-\al\beta'}{1-\beta\al'} \right)\label{propsix}\\
T(\al,\beta,\gamma)^{-1}&=&\frac{\gamma}{\al\beta+\gamma}\, 
T\left(-\frac{\al}{\gamma},-\frac{\beta}{\gamma},\frac{1}{\gamma}\right)\label{propseven}
\end{eqnarray}
These hold whenever the denominators are non-vanishing.

\subsection{Commuting families and addition formulas}\label{comapp}

Using the formulas above, it is easy to derive the following:

\begin{thm}
The following family $\{T_{s,t}(\al)\}_{a\in \C}$ of infinite matrices commute among 
themselves for any fixed values of $s$ and $t$:
$$T_{s,t}(\al)=T(\al,s \al,1-t\al)$$
\end{thm}

We also have the following ``addition" formula:
$$ T_{s,t}(\al)T_{s,t}(\al')= \frac{1}{1-s \al\al'}\, T_{s,t}\left(\frac{\al+\al'-t \al \al'}{1-s \al\al'}\right) $$

For $s=0$, we obtain a family of commuting lower triangular matrices 
$$L_t(\al)=T_{0,t}(\al)=L\left(\frac{1}{\al}-t,\al\right)$$
with the ``addition" formula:
$$L_t(\al)L_t(\al')=L_t(\al+\al'-t \al\al')$$
Changing variables from $\al$ to $a$, with $\al=\frac{1-e^{-ta}}{t}$, and writing $\ell_t(a)=L_t(\al)$
we finally get the addition formula:
\begin{equation}\label{pseudoexp} \ell_t(a)\ell_t(a')=\ell_t(a+a') \end{equation}
from which we deduce that $\ell_t(a)$ is an infinite matrix exponential. More precisely,
let $M_t$ be the infinite matrix generated by
$$ f_{M_t}=\frac{(t v-1)u}{(1-u v)^2}  \qquad M_t=\begin{pmatrix}  
0 & 0 & 0 & \cdots \\
-1 & t & 0 & \cdots \\
0 & -2 & 2 t & \cdots \\
0 & 0 & -3 & 3 t \\
\vdots & & & \ddots  \ddots
\end{pmatrix}$$
then we have $\ell_t(a)=\exp(-a M_t)$, which by triangularity holds for any finite truncation as well.
A similar analysis holds for $T_{s,t}(\al)$, but only for the infinite matrix. Assuming that $t^2-4s>0$,
and introducing another parameter $r=\sqrt{1-4s/t^2}$, the relevant change of variables
is:
$$ \al= \frac{2(e^{r t a}-1)}{t  (te^{r t a}(r+1)+r-1)} \qquad \tau_{r,t}(a)=T_{s,t}(\al)$$
in terms of which
$$ \tau_{r,t}(a)\tau_{r,t}(a')=\frac{1}{1-
\frac{(1-r^2)(e^{r t a}-1)(e^{r t a'}-1)}{(te^{rta}(r+1)+r-1)(te^{rta'}(r+1)+r-1)}
}\, \tau_{r,t}(a+a')$$
This reduces to \eqref{pseudoexp} when $r=1$ (corresponding to $s=0$).

\subsection{Truncated determinants}\label{detapp}

For any infinite matrix $A=(a_{i,j})_{i,j\in \Z_+}$, we denote by $A^{[0,n-1]}$ the finite $n\times n$
truncation of $A$ to its $n$ first rows and columns, namely the matrix with entries: $A^{[0,n-1]}=(a_{i,j})_{i,j\in [0,n-1]}$.

In general, the matrix product does not respect truncation. However, if $L,U$ are respectively lower and
upper triangular infinite matrices, then $(LU)^{[0,n-1]}=L^{[0,n-1]} U^{[0,n-1]}$. Note that this does not hold for $(UL)$
(see the example below).

Let us now examine truncated determinants, namely the determinant of such finitely truncated matrices.
By triangularity it is immediate to compute:
\begin{equation}\label{detL} \det\left(L(\al,\beta)^{[0,k]}\right)=(\al \beta)^{k(k+1)/2}=\det\left(U(\al,\beta)^{[0,k]}\right)
\end{equation}
and by the above property we deduce from \eqref{propfour} and \eqref{detL} that:
$$\det\left( T^{[0,k]}(\beta,\beta',\beta\beta'(\al\al'-1))\right)= \det\left(L(\al,\beta)^{[0,k]}U(\al',\beta')^{[0,k]}\right)=
(\al\beta\al'\beta')^{k(k+1)/2} $$
and more generally
$$\det\left( T^{[0,k]}(\al,\beta,\gamma)\right)= (\al\beta+\gamma)^{k(k+1)/2}$$
whereas
\begin{eqnarray*}
\det\left(\big(U(\al',\beta')L(\al,\beta)\big)^{[0,k]}\right)&=&
\frac{\det\left( T^{[0,k]}\left(\frac{\al'\beta\beta'}{1-\beta\beta'},
\frac{\al\beta\beta'}{1-\beta\beta'},\frac{\al\al'\beta\beta'}{1-\beta\beta'}\right)\right)}{(1-\beta\beta')^{k+1}}\\
&=& \frac{(\al\beta\al'\beta')^{k(k+1)/2}}{(1-\beta\beta')^{k+1}} 
\end{eqnarray*}
by use of \eqref{propfive}. The discrepancy with the $LU$ result is because the matrix product
now involves all the elements of the $(k+1)\times$ infinite and infinite $\times (k+1)$ rectangular matrices 
of the truncated product.

The main property allowing for proving truncated determinant identities from relations between generating
functions is the following:

\begin{lemma}\label{unitri}
Let $L,U,A$ be respectively lower triangular, upper triangular and arbitrary infinite matrices,
and let $M=LAU$.
Then:
$$ M^{[0,k]}=L^{[0,k]} \, A^{[0,k]} \, U^{[0,k]} $$
\end{lemma}

Assuming further that both $L$ and $U$ are unitriangular, we then deduce that $\det(M^{[0,k]})=\det(A^{[0,k]})$
for all $k\geq 0$. So we will have identity between all truncated determinants of two infinite matrices
$M$ and $A$ if there is a relation $M=LAU$ for $L,U$ lower and upper unitriangular infinite matrices.

\end{appendix}

\medskip

\begin{flushleft}
Institut de Physique Th\'eorique du Commissariat \`a l'Energie Atomique, \\
Unit\'e de Recherche associ\'ee du CNRS,\\
CEA Saclay/IPhT/Bat 774, F-91191 Gif sur Yvette Cedex, \\
FRANCE. \\
E-mail: {\tt philippe.di-francesco@cea.fr}
%
\end{flushleft}
\begin{flushright}
\end{flushright}

\end{document}